\documentclass[twoside,leqno,10pt, A4]{amsart}
\usepackage{amsfonts}
\usepackage{amsmath}
\usepackage{amscd}
\usepackage{amssymb}
\usepackage{amsthm}
\usepackage{amsrefs}
\usepackage{latexsym}
\usepackage{mathrsfs}
\usepackage{bbm}
\usepackage{amscd}
\usepackage{amssymb}
\usepackage{amsthm}
\usepackage{amsrefs}
\usepackage{latexsym}
\usepackage{mathrsfs}
\usepackage{bbm}
\usepackage{enumerate}
\usepackage{graphicx}
\usepackage{color}
\begin{document}

\newtheorem{theorem}[subsection]{Theorem}
\newtheorem{proposition}[subsection]{Proposition}
\newtheorem{lemma}[subsection]{Lemma}
\newtheorem{corollary}[subsection]{Corollary}
\newtheorem{conjecture}[subsection]{Conjecture}
\newtheorem{prop}[subsection]{Proposition}
\newtheorem{defin}[subsection]{Definition}

\numberwithin{equation}{section}
\newcommand{\mr}{\ensuremath{\mathbb R}}
\newcommand{\mc}{\ensuremath{\mathbb C}}
\newcommand{\dif}{\mathrm{d}}
\newcommand{\intz}{\mathbb{Z}}
\newcommand{\ratq}{\mathbb{Q}}
\newcommand{\natn}{\mathbb{N}}
\newcommand{\comc}{\mathbb{C}}
\newcommand{\rear}{\mathbb{R}}
\newcommand{\prip}{\mathbb{P}}
\newcommand{\uph}{\mathbb{H}}
\newcommand{\fief}{\mathbb{F}}
\newcommand{\majorarc}{\mathfrak{M}}
\newcommand{\minorarc}{\mathfrak{m}}
\newcommand{\sings}{\mathfrak{S}}
\newcommand{\fA}{\ensuremath{\mathfrak A}}
\newcommand{\mn}{\ensuremath{\mathbb N}}
\newcommand{\mq}{\ensuremath{\mathbb Q}}
\newcommand{\half}{\tfrac{1}{2}}
\newcommand{\f}{f\times \chi}
\newcommand{\summ}{\mathop{{\sum}^{\star}}}
\newcommand{\chiq}{\chi \bmod q}
\newcommand{\chidb}{\chi \bmod db}
\newcommand{\chid}{\chi \bmod d}
\newcommand{\sym}{\text{sym}^2}
\newcommand{\hhalf}{\tfrac{1}{2}}
\newcommand{\sumstar}{\sideset{}{^*}\sum}
\newcommand{\sumprime}{\sideset{}{'}\sum}
\newcommand{\sumprimeprime}{\sideset{}{''}\sum}
\newcommand{\sumflat}{\sideset{}{^\flat}\sum}
\newcommand{\shortmod}{\ensuremath{\negthickspace \negthickspace \negthickspace \pmod}}
\newcommand{\V}{V\left(\frac{nm}{q^2}\right)}
\newcommand{\sumi}{\mathop{{\sum}^{\dagger}}}
\newcommand{\mz}{\ensuremath{\mathbb Z}}
\newcommand{\leg}[2]{\left(\tfrac{#1}{#2}\right)}
\newcommand{\muK}{\mu_{\omega}}
\newcommand{\thalf}{\tfrac12}
\newcommand{\lp}{\left(}
\newcommand{\rp}{\right)}
\newcommand{\Lam}{\Lambda_{[i]}}
\newcommand{\lam}{\lambda}
\newcommand{\af}{\mathfrak{a}}
\newcommand{\sw}{S_{[i]}(X,Y;\Phi,\Psi)}
\newcommand{\lz}{\left(}
\newcommand{\pz}{\right)}
\newcommand{\bfrac}[2]{\lz\frac{#1}{#2}\pz}
\newcommand{\odd}{\mathrm{\ primary}}
\newcommand{\even}{\text{ even}}
\newcommand{\res}{\mathrm{Res}}
\newcommand{\sumn}{\sumstar_{(c,1+i)=1}  w\left( \frac {N(c)}X \right)}
\newcommand{\lab}{\left|}
\newcommand{\rab}{\right|}
\newcommand{\Go}{\Gamma_{o}}
\newcommand{\Ge}{\Gamma_{e}}
\newcommand{\M}{\widehat}
\newcommand{\ch}{\mathrm{conv}}

\theoremstyle{plain}
\newtheorem{conj}{Conjecture}
\newtheorem{remark}[subsection]{Remark}

\makeatletter
\def\widebreve{\mathpalette\wide@breve}
\def\wide@breve#1#2{\sbox\z@{$#1#2$}%
     \mathop{\vbox{\m@th\ialign{##\crcr
\kern0.08em\brevefill#1{0.8\wd\z@}\crcr\noalign{\nointerlineskip}%
                    $\hss#1#2\hss$\crcr}}}\limits}
\def\brevefill#1#2{$\m@th\sbox\tw@{$#1($}%
  \hss\resizebox{#2}{\wd\tw@}{\rotatebox[origin=c]{90}{\upshape(}}\hss$}
\makeatletter

\title[Moment of quadratic Hecke $L$-Functions]{First Moment of Quadratic Hecke $L$-Functions with Lower Order Term}

\author[P. Gao]{Peng Gao}
\address{School of Mathematical Sciences, Beihang University, Beijing 100191, China}
\email{penggao@buaa.edu.cn}

\author[L. Zhao]{Liangyi Zhao}
\address{School of Mathematics and Statistics, University of New South Wales, Sydney NSW 2052, Australia}
\email{l.zhao@unsw.edu.au}

\begin{abstract}
We evaluate the first moment of the family of primitive quadratic Hecke $L$-functions in the Gaussian field using the method of double Dirichlet series under the Riemann hypothesis and the Lindel\"of hypothesis.  We obtain asymptotic formulas with secondary main terms and error terms of size that is one quarter of that of the main term.
\end{abstract}

\maketitle



\section{Introduction}\label{sec 1}

M. Jutila \cite{Jutila} obtained asymptotic formulas for the first and second moments of central values of the family of primitive quadratic Dirichlet $L$-functions.  As a corollary, he showed that there are infinitely many such $L$-functions that do not vanish at the central point.  Subsequently, improvements on the error term for the first moment result of Jutila were obtained by D. Goldfeld and J. Hoffstein \cite{DoHo}, A. I. Vinogradov and L. A. Takhtadzhyan  \cite{ViTa} and M. P. Young \cite{Young1}. \newline

Using the method of double Dirichlet series, it is shown implicitly in \cite{DoHo} that for any $\varepsilon>0$, there is a linear polynomial $P(x)$ so that
\begin{align}
\label{firstmomentDasymp}
	\sum_{\substack{d\geq 1\\ d \text{\ odd}}} \mu^2(d) L \lz \tfrac 12,\leg{8d}{\cdot}\pz \Phi (\tfrac dX)  = XP(\log X)+O(X^{1/2+\varepsilon}).
\end{align}
Here $\mu$ denotes the M\"obius function, $\leg{8d}{\cdot}$ is the Kronecker symbol, and where $\Phi$ is a non-negative smooth function compactly supported on $(0, \infty)$.   The same result was also obtained in \cite{Young1} using a recursive argument. \newline

  In \cite{Florea17},  A. Florea studied a first moment analogue to the left-hand side of \eqref{firstmomentDasymp} over function fields.  She obtained an error term of size that is essentially $1/4$ power of the primary main term and exhibited a secondary main term whose magnitude is $1/3$ power of the primary main term.  Thus informed by her result, a similar result is expected to hold for the number fields case as well. Indeed, this was shown to be the case by M. \v Cech \cite{Cech3}, who proved that under the Riemann hypothesis and the generalized Lindel\"of hypothesis, for some linear polynomial $P_1(x), P_2(x)$ and any $\varepsilon>0$,
\begin{align*}
	\sum_{\substack{d\geq 1\\ d \text{\ odd}}} \mu^2(d) L\lz \tfrac 12,\leg{8d}{\cdot}\pz \Phi(\tfrac dX) = XP_1(\log X)+X^{1/3}P_2(\log X)+O(X^{1/4+\epsilon}).
\end{align*}
 Note that, if one discards the $\mu^2(d)$ on the left-hand side of \eqref{firstmomentDasymp} (that is, the sum runs over all odd $d$), then it is shown implicitly in \cite{Blomer11} (see also \cite{G&Zhao23-03}) that an asymptotical formula similar to \eqref{firstmomentDasymp} holds, with an error term of size $O(X^{1/4+\varepsilon})$. This evinces a major difference in the first moment results between the quadratic family over all conductors and that over square-free conductors.  An account for this distinction is given in \cite{Cech3}. \newline

  The aim of this paper is to study the first moment of the family of quadratic Hecke $L$-functions in the Gaussian field using the method of M. \v Cech in \cite{Cech3}.  To state our result, we need to introduce some notations.  We write $K=\mq(i)$ for the Gaussian field and $\mathcal{O}_K=\mz[i]$ for its ring of integers throughout the paper.  Denote by $N(n)$ the norm of any $n \in K$ and $\zeta_K(s)$ the Dedekind zeta function of $K$.  It is shown in Section \ref{sec2.4} below that every ideal in $\mathcal{O}_K$ co-prime to $2$ has a unique generator congruent to $1$ modulo $(1+i)^3$ which is called primary. Note that $(1+i)$ is the only prime ideal in $\mathcal{O}_K$ that lies above the integral ideal $(2)$ in $\mz$ and an element $n \in \mathcal{O}_K$ is said to be odd if $(n,1+i)=1$.  An element $\varpi \in \mathcal{O}_K$ is prime if $(\varpi)$ is a prime ideal and an element $n \in \mathcal{O}_K$ is square-free if no prime squared divides $n$. We shall henceforth reserve the letter $\varpi$ for a prime in $\mathcal O_K$.  Also, as usual, $\varepsilon$ denotes an arbitrarily small positive real number whose value may vary at each appearance. \newline

Let $\chi_m$ be the quadratic symbol $\left(\frac{m}{\cdot} \right)$ defined in Section \ref{sec2.4}, which can be viewed as an analogue in $K$ to the Kronecker symbol. As $\chi_m$ equals $1$ on the group of units of $K$,  we may regard it as a quadratic Hecke character of trivial infinite type (see Section \ref{sec2.4} for the definition and a detailed discussion) and denote the associated $L$-function by $L(s, \chi_m)$.  Furthermore, we use the notation $L^{(c)}(s, \chi_m)$ for the Euler product defining $L(s, \chi_m)$ but omitting the factors from those primes dividing $c$. It is shown in Section \ref{sec2.4} below that the quadratic Hecke character $\chi_{(1+i)d}$ is a primitive Hecke character modulo $(1+i)^5d$ of trivial infinite type for any odd, square-free $d$. \newline

 Let $\Phi$ be a non-negative smooth function compactly supported in $[1/2, 2]$. We are interested in the smoothed first moment of the family of quadratic Hecke $L$-functions given by
\begin{align}
\label{ZetWithallCharacters}
  \sumstar_{\substack{d\odd}}L(s,\chi_{(1+i)^5d})\Phi(\tfrac {N(d)}X),
\end{align}
 where $\sum^*$ indicates a sum over square-free elements in $\mathcal{O}_K$. \newline

In this paper, our main interest is to asymptotically evaluate the sum in \eqref{ZetWithallCharacters} under the Lindel\"of hypothesis for Hecke $L$-functions.  This hitherto unresolved conjecture asserts that, for any primitive Hecke character $\chi$ of trivial infinite type of conductor $q$ and arbitrary $\sigma\geq 1/2$, $t\in\mr$, we have, 
\begin{equation*}
\label{eq:LH}
	|L(\sigma+it,\chi)|\ll_{t}N(q)^{\varepsilon}.
\end{equation*}
The above estimation applies to non-primitive characters as well by writing the corresponding $L$-function in terms of the primitive one. \newline

We are now ready to state our main result.
\begin{theorem}
\label{Theorem for all characters}
 With the notation as above and the Lindel\"of hypothesis, let $\Phi(t)$ be a non-negative smooth function that is compactly supported in $[1/2, 2]$ with Mellin transform $\widehat \Phi(s)$.  Set 
\begin{align}
\label{betadef}
 \beta:=\sup\{\Re(\rho):\zeta_K(\rho)=0\}.
\end{align}
Fix $s\in\mc$ with $1/3<\Re(s)<1$, $s\neq 1/2$. Then for any $\varepsilon>0,$
		\begin{equation}
\label{eq:general moment result}
\begin{aligned}
				\sumstar_{d\odd}L(s,\chi_{(1+i)^5d})\Phi (\tfrac {N(d)}X)= X \widehat\Phi(1)& R_{K,1} (s)+ X^{3/2-s}\widehat\Phi(\tfrac32-s)X_K(s)R_{K,1}(1-s)\\
				&+X^{1/2- s/3} \widehat\Phi(\tfrac12-\tfrac s3)R_{K,2}(s)+X^{(2-2s)/3}\widehat\Phi(\tfrac{2-2s}{3})X_K(s)R_{K,2}(1-s)\\
				&+O_\varepsilon\lz X^{(1-\Re(s))/2+\varepsilon}+X^{\beta/2+\epsilon}+X^{(\beta+1)/2-\Re(s)+\varepsilon}\pz,
\end{aligned}
		\end{equation} 
		where
\begin{align}
\label{XK}
	X_K(s):=\bfrac{\pi^2}{2^5}^{s-1/2}\frac{\Gamma (1-s)}{\Gamma(s)},
\end{align}	
  and where $R_{K,1}(s), R_{K,2}(s)$ are explicitly defined in Theorem \ref{thm:MDS}.
\end{theorem}

  Upon taking the limit $s\rightarrow 1/2$ in \eqref{eq:general moment result} and note that the error term there is uniformly bounded in the process, we obtain an asymptotical formula for the first moment of central values of the family of primitive quadratic Hecke $L$-functions.
\begin{corollary}
\label{Thmfirstmomentatcentral}
	With the same notation and assumptions as Theorem~\ref{Theorem for all characters}, we have, 
\begin{align*}
\begin{split}			
	& \sumstar_{d\odd}L( \tfrac12,\chi_{(1+i)^5d})\Phi(\tfrac {N(d)}X) = XQ_1(\log X)+X^{1/3}Q_2(\log X)+O_\varepsilon(X^{\beta/2+\varepsilon}),
\end{split}
\end{align*}
  where $Q_1$, $Q_2$ are linear polynomials whose coefficients involve only absolute constants and $\M \Phi(1)$ and $\M \Phi'(1)$.
\end{corollary}

  We shall not be too concerned with the explicit expressions of $Q_1$ and $Q_2$ here since our main focus is the $O$-term. Explicit expressions can be obtained using the arguments given in the proof of \cite[Theorem 1.2, (2)]{Cech3}. The generalized Riemann hypothesis (GRH) yields that $\beta=1/2$, with $\beta$ defined in \eqref{betadef}. As the Lindel\"of hypothesis is a consequence of GRH,  we readily deduce from Corollary \ref{Thmfirstmomentatcentral} the following result.
\begin{corollary}
\label{Thmfirstmomentatcentral1}
	With the notation as above and the truth of GRH, we have, 
\begin{align*}
\begin{split}			
	& \sumstar_{d\odd}L(\tfrac12,\chi_{(1+i)^5d})\Phi (\tfrac {N(d)}X) = XQ_1(\log X)+X^{1/3}Q_2(\log X)+O_\varepsilon(X^{1/4+\varepsilon}).
\end{split}
\end{align*}
\end{corollary}

To prove Theorem \ref{Theorem for all characters}, we first apply the Mellin inversion so that
\begin{align}
\label{Sumintegral}
\begin{split}			
	& \sumstar_{d\odd}L(s,\chi_{(1+i)^5d})\Phi (\tfrac {N(d)}X) = \frac{1}{2\pi i}\int\limits_{(c)}A(s,w)\widehat  \Phi(s) X^s \dif s,
 \end{split}
\end{align}
 where $c>2$ and 
 \begin{align}
\label{Asw}
\begin{split}	
 	A(s,w)=\sumstar_{d\odd}\frac{L(s,\chi_{(1+i)^5d})}{N(d)^w}.
  \end{split}
\end{align}

Thus we are lead to study the analytic properties of $A(s,w)$. We shall do so, using the approach of \v Cech in \cite{Cech3}, by treating $A(s,w)$ as a double Dirichlet series.  A key ingredient in \cite{Cech3} is a novel functional equation established in \cite[Proposition 2.3]{Cech1} for Dirichlet $L$-function attached to a general (not necessarily primitive) Dirichlet character. Similarly, we need such functional equations for quadratic Hecke $L$-functions, available to us from \cite[Proposition 2.5]{G&Zhao2024-1}. With the aid of this, we are able to establish the following result which gives some analytical properties of $A(s,w)$.
\begin{theorem}
\label{thm:MDS}
	With the notation as above and assuming the truth of the Lindel\"of hypothesis, we have the following properties for $A(s,w)$.
\begin{enumerate}
		\item The double Dirichlet series $A(s,w)$ has meromorphic continuation to the region
	$$
		T=\{(s,w) \in \comc^2 : \Re(s+w)>1/2,\ \Re(s+2w)>1,\ \Re(w)>0\}.
	$$
The possible poles of $\zeta_K(2w)\zeta_K(2s+2w-1)A(s,w)$ in $T$ occur at $w=1$, $s+w=3/2$, $s+(2j+1)w=3/2$ and $2js+(2j+1)w=j+1$ for $j\in\mathbb{N}$.  Moreover, $A(s,w)$ is polynomially bounded in vertical strips in $T$. 
		\item 
			We have
\[ R_{K,1}(s):=\res_{w=1}A(s,w)=\frac {\pi\zeta_K^{(2)}(2s)B_K(s)}{6\zeta_K(2)}, \quad \mbox{where} \quad B_K(s)=\prod_{(\varpi, 2)=1}\lz1-\frac{1}{N(\varpi)^{2s}(N(\varpi)+1)}\pz \]
is an Euler product that converges absolutely for $\Re(s)>0$.
		\item We have
		$\res_{w=3/2-s}A(s,w)=X_K(s)\res_{w=1}A(1-s,w).$
		\item We have
		$\res_{w=2/3-2s/3} A(s,w)=X_K(s)\res_{w=1/2-s/3}A(1-s,w).$
\item We have
$$R_{K,2}(s):=\res_{w=1/2-s/3} A(s,w)=Y_K(s)\zeta_K(2s)Q_K(s),$$
 where
\begin{align}
\label{Ydef}
\begin{split}
		Y_K(s):=& Y(\tfrac {2s}3)Y(\tfrac 12-\tfrac s3), \quad Y(s):= 2^{-2-2s}\pi^{-(1-2s)}\frac {\Gamma(1-s)}{\Gamma(s)}, \\
		Q_K(s):=& \frac{\pi E_K\bfrac{2s}3 2^{2+2s}}{\zeta_K\bfrac{4s}3\zeta_K(2)} \lz1-\frac{1}{2^{2s}}\pz \lz1-\frac 1{2^{2(1/2-s/3)}} \pz^{-1}\lz1+\frac{2}{2^{2s/3}(2-2^{2s/3})}\pz^{-1} \frac{1}{2^{1/2+s/3}} \lz \lz 1-\frac{1}{2^{1+2s/3}}\pz^{-1}-2 \pz \\
		&  \times \prod_{\substack{\varpi \\ \varpi \odd}}\lz 1+\frac{N(\varpi)^{1+4s/3}+N(\varpi)^{1+2s/3}-N(\varpi)-N(\varpi)^{2s}}{N(\varpi)^{2s/3}(N(\varpi)^{1+2s/3}+N(\varpi)-N(\varpi)^{4s/3})(N(\varpi)^{1+4s/3}-1)}\pz,
\end{split}
\end{align}
 and 
\begin{align}
\label{EKdef}
\begin{split}
	E_K(s)=\prod_{\varpi}\lz 1+\frac{N(\varpi)^s-N(\varpi)^{1-s}+1}{(N(\varpi)+1)(N(\varpi)^s+1)(N(\varpi)^{1-s}-1)}\pz.
\end{split}
\end{align}
\end{enumerate}
\end{theorem}

The notion of polynomial boundedness in vertical strips in our setting is articulated in Definitions \ref{def:bounded} and \ref{def:bounded meromorphic}. Now, the deduction of Theorem \ref{Theorem for all characters} from Theorem \ref{thm:MDS} follows well-known standard methods, i.e. by shifting the contour of integration
in \eqref{Sumintegral} and picking up the residues from the poles in the process, as in \cite{Cech3}.  We will therefore only give a sketch of the proof of Theorem~ \ref{Theorem for all characters}.

\section{Preliminaries}
\label{sec 2}

\subsection{Quadratic residue symbol and quadratic Hecke character}
\label{sec2.4}
   Recall that $K=\mq(i)$ which has class number one.  We write $U_K=\{ \pm 1, \pm i \}$ and $D_{K}=-4$ for the group of units in $\mathcal{O}_K$ and the discriminant of $K$, respectively.  Note that $n$ is square-free if and only if $\mu_{[i]}(n) \neq 0$, where $\mu_{[i]}$ is the M\"obius function on $\mathcal{O}_K$.  \newline

   Every ideal in $\mathcal{O}_K$ co-prime to $2$ has a unique generator congruent to $1$ modulo $(1+i)^3$.  This generator is called primary. It follows from Lemma
6 on \cite[p.121]{I&R} that an element $n=a+bi \in \mathcal{O}_K$ with $a, b \in \mz$ is primary if and only if $a \equiv 1 \pmod{4}, b \equiv
0 \pmod{4}$ or $a \equiv 3 \pmod{4}, b \equiv 2 \pmod{4}$.  We shall refer to these two cases above as $n$ is of type $1$ and type $2$, respectively. \newline

 For $q \in \mathcal{O}_K$, let $\left (\mathcal{O}_K / (q) \right )^*$ denote the group of reduced residue classes modulo $q$, i.e., the multiplicative group of invertible elements in $\mathcal{O}_K / (q)$.  For $0 \neq q \in \mathcal{O}_K$, let $\chi$ be a homomorphism:
\begin{align*}
  \chi: \left (\mathcal{O}_K / (q) \right )^*  \rightarrow \{ z \in \mc :  |z|=1 \}.
\end{align*}
     Following the nomenclature of \cite[Section 3.8]{iwakow}, we shall refer $\chi$ as a Dirichlet character modulo $q$. We say that such a Dirichlet character $\chi$ modulo $q$ is primitive if it does not factor through $\left (\mathcal{O}_K / (q') \right )^*$ for any divisor $q'$ of $q$ with $N(q')<N(q)$. \newline

For a prime $\varpi \in \mathcal{O}_{K}$
with $N(\varpi) \neq 2$, the quadratic symbol is defined for $a \in
\mathcal{O}_{K}$, $(a, \varpi)=1$ by $\leg{a}{\varpi} \equiv
a^{(N(\varpi)-1)/2} \pmod{\varpi}$, with $\leg{a}{\varpi} \in \{
\pm 1 \}$.  If $\varpi | a$, we define
$\leg{a}{\varpi} =0$.  Then the quadratic symbol is extended
to any composite odd $n$  multiplicatively. We further define $\leg {\cdot}{c}=1$ for $c \in U_K$.  Let $\chi_c$ stand for the symbol $\leg {c}{\cdot}$, where we define $\leg {c}{n}=0$ when $1+i|n$. \newline

  The following quadratic reciprocity law (see \cite[(2.1)]{G&Zhao4}) holds for two co-prime primary elements $m$, $n \in \mathcal{O}_{K}$:
\begin{align}
\label{quadrec}
 \leg{m}{n} = \leg{n}{m}.
\end{align}
Moreover, from Lemma 8.2.1 and Theorem 8.2.4 in \cite{BEW}, we deduce the following supplementary laws hold for primary $n=a+bi$ with $a, b \in \mz$:
\begin{align}
\label{supprule}
  \leg {i}{n}=(-1)^{(1-a)/2} \qquad \mbox{and} \qquad  \hspace{0.1in} \leg {1+i}{n}=(-1)^{(a-b-1-b^2)/4}.
\end{align}

The quadratic symbol $\leg {\cdot}{n}$ defined above is a Dirichlet character modulo $n$. It follows from \eqref{supprule} that the quadratic symbol $\chi_i:=\leg {i}{\cdot}$ defines a primitive Dirichlet character modulo $4$.  Also,  \eqref{supprule} implies that  the quadratic symbol $\chi_{1+i}:=\leg {1+i}{\cdot}$ defines a primitive Dirichlet character modulo $(1+i)^5$ (see more details in \cite[Section 2.1]{G&Zhao2023}). Furthermore, we observe that there exists a primitive quadratic Dirichlet character $\psi_2$ modulo $2$ since $\left (\mathcal{O}_K / (2) \right )^*$ is isomorphic to the cyclic group of order two generated by $i$.  Then we have $\psi_2(n)=-1$ for $n \equiv i \pmod 2$. \newline

   For any $l \in \mz$ with $4 | l$, we define a unitary character $\chi_{\infty}$ from $\mc^*$ to $S^1$ by:
\begin{align*}
  \chi_{\infty}(z)=\leg {z}{|z|}^l.
\end{align*}
  The integer $l$ is called the frequency of $\chi_{\infty}$. \newline

  Now, for a given Dirichlet character $\chi$ modulo $q$ and a unitary character $\chi_{\infty}$, we can define a Hecke character $\psi$ modulo $q$ (see \cite[Section 3.8]{iwakow}) on the group of fractional ideals $I_K$ in $K$, such that for any $(\alpha) \in I_K$,
\begin{align*}
  \psi((\alpha))=\chi(\alpha)\chi_{\infty}(\alpha).
\end{align*}
If $l=0$, we say that $\psi$ is a Hecke character modulo $q$ of trivial infinite type. In this case, we may regard $\psi$ as defined on $\mathcal{O}_K$ instead of on $I_K$, setting $\psi(\alpha)=\psi((\alpha))$ for any $\alpha \in \mathcal{O}_K$. We may also write $\chi$ for $\psi$ as well, since we have $\psi(\alpha)=\chi(\alpha)$ for any $\alpha \in \mathcal{O}_K$. We then say such a Hecke character $\chi$ is primitive modulo $q$ if $\chi$ is a primitive Dirichlet character. Likewise, we say that  $\chi$ is induced by a primitive Hecke character $\chi'$ modulo $q'$ if $\chi(\alpha)=\chi'(\alpha)$ for all
$(\alpha, q')=1$. \newline

   We now define an abelian group $\text{CG}$ such that it is generated by three primitive quadratic Hecke characters of trivial infinite type with corresponding moduli dividing $(1+i)^5$.  More precisely,
\begin{align*}
  \text{CG}=\{ \chi_j : j=1, i, 1+i, i(1+i) \},
\end{align*}
and the commutative binary operation on $\text{CG}$ is given by $\chi_i \cdot \chi_{i(1+i)}=\chi_{1+i}$, $\chi_{1+i} \cdot \chi_{i(1+i)}=\chi_i$ and $\chi_j \cdot \chi_j=\chi_1$ for any $j$. As we shall only evaluate the related characters at primary elements in $\mathcal{O}_K$, the definition of such a product is therefore justified. \newline

  We note that the product of $\chi_c$ for any primary $c$ with any $\chi_j \in \text{CG}$ gives rise to a Hecke character of trivial infinite type. A Hecke character $\chi$ is said to be primitive modulo $q$ if it does not factor through $\left (\mathcal{O}_K / (q') \right )^{\times}$ for any  divisor $q'$ of $q$ with $N(q')<N(q)$. To determine the primitive Hecke character that induces such a product, we observe that every primary $c$ can be written uniquely as
\begin{align*}
 c=c_1c_2, \quad \text{$c_1, c_2$ primary and $c_1$ square-free}.
\end{align*}
The above decomposition allows us to conclude that if $c_1$ is of type 1, then $\chi_c \cdot \chi_j$ for $j \in \{1, i, 1+i, i(1+i)\}$ is induced by the primitive Hecke character $\leg{\cdot}{c_1} \cdot \chi_j$ with modulus $c_1$, $4c_1$, $(1+i)^5c_1$ and $(1+i)^5c_1$ for $j=1, i, 1+i$ and $i(1+i)$, respectively. This is because that $\leg {\cdot }{c_1}$ is trivial on $U_K$ by \eqref{supprule}. Similarly, if $c_1$ is of type 2,  $\chi_c \cdot \chi_j$ for $j \in \{1, i, 1+i, i(1+i) \}$ is induced by the primitive Hecke character $\chi_j \cdot \psi_2 \cdot \leg {\cdot}{c_1}$ with modulus $2c_1, 4c_1, (1+i)^5c_1$ and $(1+i)^5c_1$ for $j=1, i, 1+i$ and $i(1+i)$, respectively. In particular, the above implies that the symbol $\chi_{(1+i)^5d}$ defines a primitive quadratic Hecke character modulo $(1+i)^5d$ of trivial infinite type for any odd and square-free $d \in \mathcal{O}_K$, which is shown in \cite[ Section 2.1  ]{G&Zhao4}.  We make the convention that for any primary square-free $n \in \mathcal O_K$, we shall use $\chi_n\cdot \chi_j$ for any $\chi_j \in \text{CG}$ to denote the corresponding primitive Hecke character $\chi$ that induces it throughout the paper.

\subsection{Functional equations for quadratic Hecke $L$-functions}
   For any complex number $z$, let
   \begin{align*}
 \widetilde{e}(z) =\exp \left( 2\pi i  \left( \frac {z}{2i} - \frac {\bar{z}}{2i} \right) \right) .
\end{align*}
  For any $r\in \mathcal{O}_K$, we define the quadratic Gauss sum $g(r, \chi)$ associated to any quadric Hecke character $\chi$ modulo $q$ of trivial infinite type and the quadratic Gauss sum $g(r, n)$ associated to the quadratic residue symbol $\leg {\cdot}{n}$ for any $(n,2)=1$ by
\begin{align*}
 g(r,\chi) = \sum_{x \bmod{q}} \chi(x) \widetilde{e}\leg{rx}{q}, \quad g(r,n) = \sum_{x \bmod{n}} \leg{x}{n} \widetilde{e}\leg{rx}{n}.
\end{align*}
In the case $r=1$, we simply write $g(\chi)$ for $g(1, \chi)$ and $g(n)$ for $g(1, n)$. \newline

  Let $\varphi_{[i]}(n)$ be the number of elements in $(\mathcal{O}_K/(n))^*$.  We recall from \cite[Lemma 2.2]{G&Zhao4} the following explicit evaluations of $g(r,n)$ for primary $n$.
\begin{lemma}
\label{Gausssum}
\begin{enumerate}[(i)]
\item  We have
\begin{align*}
 g(rs,n) & = \overline{\leg{s}{n}} g(r,n), \qquad (s,n)=1, \\
   g(k,mn) & = g(k,m)g(k,n),   \qquad  m,n \text{ primary and } (m , n)=1 .
\end{align*}
\item Let $\varpi$ be a primary prime in $\mathcal{O}_K$. Suppose $\varpi^{h}$ is the largest power of $\varpi$ dividing $k$. (If $k = 0$ then set $h = \infty$.) Then for $l \geq 1$,
\begin{align*}
g(k, \varpi^l)& =\begin{cases}
    0 \qquad & \text{if} \qquad l \leq h \qquad \text{is odd},\\
    \varphi_{[i]}(\varpi^l) \qquad & \text{if} \qquad l \leq h \qquad \text{is even},\\
    -N(\varpi)^{l-1} & \text{if} \qquad l= h+1 \qquad \text{is even},\\
    \leg {ik\varpi^{-h}}{\varpi}N(\varpi)^{l-1/2} \qquad & \text{if} \qquad l= h+1 \qquad \text{is odd},\\
    0, \qquad & \text{if} \qquad l \geq h+2.
\end{cases}
\end{align*}
\end{enumerate}
\end{lemma}

 For any primary $n \in \mathcal O_K$, we define the Dirichlet character $\widetilde{\chi_n}$ by $\widetilde{\chi_n}=\chi_j \cdot \leg {\cdot}{n}$ when $n$ is of type $j$ for $j=1,2$. Our discussions above imply that $\widetilde{\chi_n}$ is a Hecke character of trivial infinite type. Throughout the paper, we regard $\widetilde{\chi_n}$ as a character modulo $2n$ so that we have $\widetilde{\chi_n}(1+i)=0$. It is shown in \cite[(2.3)]{G&Zhao2024-1} that,  
\begin{align}
\label{g2exp}
 g\left( r,\widetilde{\chi_n} \right) = \leg {i}{n} \Big((-1)^{\Im (r)}+(-1)^{\Re (r)+j-1}\Big )g(r, n).
\end{align}

 Recall that $\zeta_K(s)$ denotes the Dedekind zeta function of $K$.  A well-known result of E. Hecke shows that $L(s, \chi)$ has an analytic continuation to the whole complex plane and satisfies the functional equation (see \cite[Theorem 3.8]{iwakow})
\begin{align}
\label{fneqn}
  \Lambda(s, \chi) = W(\chi)\Lambda(1-s, \chi),
\end{align}
  where
\begin{align}
\label{Lambda}
  \Lambda(s, \chi) = (|D_K|N(q))^{s/2}(2\pi)^{-s}\Gamma(s)L(s, \chi),
\end{align}
   and
\begin{align*}
  W(\chi) = g(\chi)(N(q))^{-1/2}.
\end{align*}

  As in the case of the classical quadratic Gauss sum over $\mz$, one expects that the Gauss sum  $g(\chi)$ associated to any primitive quadratic Hecke character $\chi$ modulo $q$ equals $N(q)^{1/2}$.  This is indeed the case and articulated by the following Lemma whose proof can be found in \cite[Lemma 2.3]{G&Zhao2023}. 
\begin{lemma}
\label{Gausssumevaluation}
 With notations as above, for any primary, square-free $n \in \mathcal{O}_K$, let $\chi=\chi_n \cdot \psi_j$ with $\psi_j \in \text{CG}$ and $q$ be the modulus of $\chi$.  We have
\begin{align*}
   g(\chi)=N(q)^{1/2}.
\end{align*}
\end{lemma}

From Lemma~\ref{Gausssumevaluation} that for any primary, square-free $c \in \mathcal{O}_K$ and any $\chi=\chi_c \cdot \chi_j$ with $\chi_j \in \text{CG }$, we have $W(\chi ) = 1$. We conclude from this, \eqref{fneqn} and \eqref{Lambda} that if such $\chi$ is of modulus $q$, then
\begin{align}
\label{fneqnL}
  L(1-s, \chi)=N(q)^{(2s-1)/2}\pi^{1-2s}\Gamma^{-1}(1-s)\Gamma(s)L(s, \chi).
\end{align}

    We also quote from \cite[Proposition 2.5]{G&Zhao2024-1} the following functional equation for $L(s,\widetilde{\chi_n})$ for a general $n$.
\begin{proposition}
\label{Functional equation with Gauss sums}
   With the notation as above. For any primary $n \in \mathcal O_K$, $n \neq \square $, we have
\begin{align}
\begin{split}
\label{fcneqnallchi}
  &  L(s,\widetilde{\chi_n})= N(n)^{-s}Y(s) \sum_{\substack{ k \neq 0 \\ k \in
   \mathcal{O}_K}}\frac {g(k,\widetilde{\chi_n})}{N(k)^{1-s}},
\end{split}
\end{align}
  where $Y(s)$ is defined in \eqref{Ydef}. 
\end{proposition}

   Note that if $\chi$ is primitive, an argument similar to that given in \cite[\S 9]{Da} shows that $g_K(k, \chi)=g_K(1,\chi)\overline{\chi}(k)$. Thus the functional equation given in \eqref{fcneqnallchi} is in agreement with that given in \eqref{fneqnL}.

\subsection{A mean value estimate for quadratic Hecke $L$-functions}
 In the proof of Theorem \ref{Theorem for all characters}, we need the following large sieve estimate for quadratic Hecke $L$-functions. 
\begin{lemma}\label{le:lindelof on average}

  With the notation above, we have for any $\chi \in \text{CG}, $ and any $s \in \mc$ such that $\Re(s) \geq 0$,  
		$$
		\sum_{\substack{m \odd \\ N(m) \leq X}} |L(s,\chi_{m}\chi)| \ll_{s} X^{\max\{1, 3/2-\Re(s) \}+\varepsilon}.
		$$
		In particular, for $\Re(s) \geq 0$,  
		$$
			\sum_{N(n)\leq X}\lab L\lz s, \widetilde{\chi_n} \pz\rab\ll_{s} X^{\max\{1, 3/2-\Re(s)\}+\varepsilon}.
		$$
\end{lemma}
\begin{proof}
 For any $\chi \in \text{CG}$ and $\Re(s) \geq 1/2$, we have as a consequence of \cite[Corollary 1.4]{BGL}), the following estimation on the second moment of quadratic Hecke $L$-functions.
\begin{align*}
\sumstar_{\substack{m \odd \\ N(m) \leq X}} |L(s,\chi_{m}\chi)|^2 \ll (X|s|)^{1+\varepsilon}. 
\end{align*}

   It follows from the above and H\"older's inequality that for any  $\Re(s) \geq 1/2$, 
\begin{align}
\label{L1est}
		\sumstar_{\substack{m \odd \\ N(m) \leq X}} |L(s,\chi_{m}\chi)| \ll_{s} X^{1+\varepsilon}. 
\end{align}
 
  Recall that we regard $\chi_{m}\chi$ as a primitive character when $m$ is primary and square-free, and its conductor equals $N(m)$ times a natural number not exceeding $32$. It follows that, upon using the functional equation \eqref{fneqnL}, if $0 \leq \Re(s) <1/2$, 
\begin{align}
\label{L1est1}
		\sumstar_{\substack{m \odd \\ N(m) \leq X}} |L(s,\chi_{m}\chi)| \ll \sumstar_{\substack{m \odd \\ N(m) \leq X}} |L(1-s,\chi_{m}\chi)|N(m)^{1/2-\Re(s)} \ll_{s} X^{3/2-\Re(s)+\varepsilon}. 
\end{align}

Thus, combining \eqref{L1est} and \eqref{L1est1}, for $\Re(s) \geq 0$, we get
\begin{align}
\label{L1est2}
		\sumstar_{\substack{m \odd \\ N(m) \leq X}} |L(s,\chi_{m}\chi)| \ll_{s} X^{\max\{1,3/2-\Re(s) \}+\varepsilon}. 
\end{align}

   For the general case, we write $m=m_0m_1^2$ with $m_0, m_1$ primary and $m_0$ is square-free.  This gives that (recall that we regard $\chi_{m_0}\chi$ as a primitive character) 
$$L\lz s,\chi_{m}\chi\pz=L\lz s,\chi_{m_0}\chi\pz\prod_{\varpi|2m_1}\lz1-\frac{\chi_{m_0}\chi(\varpi)}{N(\varpi)^s}\pz  \quad \text{or} \quad L\lz s,\chi_{m}\chi\pz=L\lz s,\chi_{m_0}\chi\pz\prod_{\varpi|m_1}\lz1-\frac{\chi_{m_0}\chi(\varpi)}{N(\varpi)^s}\pz , $$ 
depending on the conductor of $\chi_{m_0}\chi$. In both cases, if $\Re(s) \geq 0$, 
$$
		\lab L\lz s,\chi_{m}\chi\pz\rab \leq \lab L\lz s,\chi_{m_0}\chi\pz\rab\lab\prod_{\varpi|2m_1}\lz1-\frac{\chi_{m_0}\chi(\varpi)}{N(\varpi)^s}\pz\rab\ll |L(s,\chi_{m_0}\chi)| 2^{w_{[i]}(m_1)} \ll N(m_1)^{\varepsilon}|L(s,\chi_{m_0}\chi)|. 
	$$
Here we write $w_{[i]}(n)$ the number for primary prime divisors of $n$ and the last estimation above follows from the bound
\begin{align}
\label{wbound}
  w_{[i]}(n) \leq \frac {\log N(n)}{\log \log N(n)}\Big(\log 2+O\Big(\frac 1{\log \log N(n)} \Big)\Big ), \; \mbox{for} \; N(n) \geq 3.
\end{align}
The bound \eqref{wbound} is analogous to a result for the number of integer prime divisors of any $n \in \mz$ given in \cite[Theorem 2.11]{MVa1} and can be established in a similar way. \newline

  It follows from the above and \eqref{L1est2} that
$$
		\sum_{\substack{m \odd \\ N(m) \leq X}} |L(s,\chi_{m}\chi)|=\sum_{\substack{m_2 \odd \\ N(m_2) \leq X^{1/2}}} \ \sumstar_{\substack{m_1 \odd \\ N(m_1) \leq X/N(m_2)^2}}|L(s,\chi_{m_1}\chi)|N(m_1)^{\varepsilon} \ll_{s} X^{\max\{1,3/2-\Re(s) \}+\varepsilon}.
$$
 This proves the first assertion of the lemma. \newline

 Note that we have $\widetilde{\chi_n}=\chi_n$, so that the second assertion of the lemma follows from the first one. This completes the proof of the lemma.
\end{proof}

We also have the following Lindel\"of-type bound on average for the right-hand side expression in \eqref{fcneqnallchi}.
\begin{lemma}
\label{le:lindelof for K}
	With the notation as above, we have for any $s \in \mc$, 
\begin{align}
\label{Kest}
	 \sum_{\substack{ n \odd \\ N(n)\leq X}}	\lab\sum_{\substack{ k \neq 0 \\ k \in
   \mathcal{O}_K}}\frac {g(k,\widetilde{\chi_n})N(n)^{-1/2}}{N(k)^{s}} \rab\ll_{s} X^{{\max \{1,3/2-\Re(s)\}}+\varepsilon}.
\end{align}
\end{lemma}
\begin{proof}
	When $\Re(s)>1$, by \eqref{g2exp},
$$
	 \sum_{\substack{ k \neq 0 \\ k \in
   \mathcal{O}_K}}\frac {g(k,\widetilde{\chi_n})N(n)^{-1/2}}{N(k)^{s}} \ll \sum_{\substack{ k \odd }}\frac {\lab g(k,n)N(n)^{-1/2} \rab}{N(k)^{\Re(s)}}.
	$$

We write for any fixed primary $n$ that $k=k_1k_2$ with $k_1, k_2$ primary and $\varpi | k_1 \Rightarrow \varpi | n$ and that $(k_2, n)=1$. It follows from this and Lemma \ref{Gausssum} that
$$
	\sum_{\substack{ k \odd }}\frac {\lab g(k,n)N(n)^{-1/2} \rab}{N(k)^{\Re(s)}} \ll \sum_{\substack{ k_1, k_2 \odd }}\frac {\lab g(k_1,n)N(n)^{-1/2} \rab}{N(k_1k_2)^{\Re(s)}} \ll \sum_{\substack{ k_1\odd }}\frac {\lab g(k_1,n)N(n)^{-1/2} \rab}{N(k_1)^{\Re(s)}}.
	$$

   We further write $n=\varpi^{s_1}_1\cdots \varpi^{s_l}_l$ with $s_i >0$ for $1 \leq i \leq l$ and apply Lemma \ref{Gausssum} again to see that 
\begin{align*}
  \sum_{\substack{ k_1\odd }} & \frac {\lab g(k_1,n)N(n)^{-1/2} \rab}{N(k_1)^{\Re(s)}} \ll \prod_{1 \leq i \leq l}\Big(\sum^{\infty}_{j=0}\frac {\lab  g(\varpi_i^{j}, \varpi^{s_i}_i)N(\varpi^{s_i}_i)^{-1/2} \rab}{N(\varpi_i)^{j\Re(s)}} \Big ) \\
& \ll  \prod_{1 \leq i \leq l}\Big(\frac {\lab  g(\varpi_i^{s_i-1}, \varpi^{s_i}_i)N(\varpi^{s_i}_i)^{-1/2} \rab}{N(\varpi_i)^{(s_j-1)\Re(s)}}+\sum^{\infty}_{j=s_i}\frac {\lab  g(\varpi_i^{j}, \varpi^{s_i}_i)N(\varpi^{s_i}_i)^{-1/2} \rab}{N(\varpi_i)^{j\Re(s)}}\Big ) \\
& \ll \prod_{1 \leq i \leq l}\Big(\frac {N(\varpi_i)^{(s_i-1)/2} }{N(\varpi_i)^{(s_j-1)\Re(s)}}+\sum^{\infty}_{j=s_i}\frac {N(\varpi_i)^{s_i/2} }{N(\varpi_i)^{j\Re(s)}}\Big ) \ll 2^{w_{[i]}(n)} \ll N(n)^{\varepsilon}.
\end{align*}
Here the last bound above follows from \eqref{wbound}. \newline

  We derive from the above discussions that for $s \in \mc$ with $\Re(s)>0$,
\begin{align}
\label{gsumbound}
	 \sum_{\substack{ k \neq 0 \\ k \in
   \mathcal{O}_K}}\frac {g(k,\widetilde{\chi_n})N(n)^{-1/2}}{N(k)^{s}} \ll N(n)^{\varepsilon}.
\end{align}
 This leads to \eqref{Kest} for $\Re(s) > 1$.\newline 

For $0 \leq \Re(s)\leq 1$, we deduce from Proposition \ref{Functional equation with Gauss sums}, 
$$ \sum_{\substack{ k \neq 0 \\ k \in
   \mathcal{O}_K}}\frac {g(k,\widetilde{\chi_n})N(n)^{-1/2}}{N(k)^{s}}=N(n)^{1/2-s}4^{2-s}\pi^{-(1-2s)}\frac {\Gamma(1-s)}{\Gamma(s)}L\lz1-s,\widetilde{\chi_n} \pz.$$ 

  As $1-\Re(s)\geq 0$, we derive from Lemma \ref{le:lindelof on average} and partial summation that,
	$$
		\begin{aligned}
			\sum_{\substack{ n \odd \\ N(n)\leq X}}\lab \sum_{\substack{ k \neq 0 \\ k \in
   \mathcal{O}_K}}\frac {g(k,\widetilde{\chi_n})N(n)^{-1/2}}{N(k)^{\Re(s)}}\rab & \ll \sum_{\substack{ n \odd \\ N(n)\leq X}}\frac{\lab L\lz 1-s, \widetilde{\chi_n}\pz\rab}{N(n)^{\Re(s)-1/2}} \\
   & \ll_{s}X^{\max\{1,1/2+\Re(s)\}+\varepsilon}X^{1/2-\Re(s)} =X^{\max\{1,3/2-\Re(s)\}+\varepsilon}.
		\end{aligned}
	$$
This completes the proof of the lemma. 
\end{proof}

\subsection{Multivariable complex analysis} 

We collect in this section some results from multivariable complex analysis.  We refer the reader to \cite[Section 2.3]{Cech3} and the references therein.  First we introduce the notion of a domain of holomorphy. 
\begin{defin}
	An open set $R\subset \mc^n$ is said to be a domain of holomorphy if there are no open sets $R_1,R_2\subset\mc^n$ with the following properties.  The set $R_2$ is connected, $R_2 \not\subset R$, $\varnothing \neq R_1\subset R\cap R_2$ and for any holomorphic function $f$ on $R$, there is a function $f_2$ holomorphic on $R_2$ such that $f=f_2$ on $R_1.$ 
\end{defin}

    Next we need a notion of a tube domain.
\begin{defin}
	A set $T\subset\mc^n$ is said to be a tube domain if there is a connected  set $U\subset\mr^n$ such that $T=U+i\mr^n=\{z\in\mc^n:\ \Re(z)\in U\}.$
	The set $U$ is called the base of $T$.
\end{defin}
  	We say that a tube domain is open, bounded, compact, etc. if its base is open, bounded, compact, etc.  A tube subdomain of $T$ is a set $S\subset T$  that is also a tube domain.  We note the following generalization of Bochner's Tube Theorem \cite{Boc}.
\begin{theorem}
\label{thm:bochner}
	An open tube domain is a domain of holomorphy if and only if it is convex.	
\end{theorem}
We denote the convex hull of $T$ by $\ch(T)$. Theorem \ref{thm:bochner} implies that every holomorphic function on an open tube domain $T$ has holomorphic continuation to $\ch (T)$. The same result applies for any function meromorphic $f$ on $T$ if there are finitely many functions $g$ holomorphic on $T$ such that $f g$ is holomorphic on $T$. \newline

  We need to control the size of the meromorphic continuation of functions of two complex variables.  To this end, we need the following definition. 
\begin{defin}\label{def:bounded}
	Let $T\subset \mc^2$ be a tube domain and $f(s,w):T\rightarrow \mc$ a holomorphic function on $T$. We say that $f$ is polynomially bounded in vertical strips in $T$ if for every compact tube subdomain $S\subset T$, there are constants $\alpha,\beta$ such that for all $(s,w)\in S$, we have
	$$
	|f(s,w)|\ll (1+|s|)^{\alpha} (1+|w|)^{\beta}.
	$$
\end{defin}
  
 A similar definition applies for meromorphic functions as well.
\begin{defin}\label{def:bounded meromorphic}
	Let $T\subset \mc^2$ be a tube domain and $f(s,w):T\rightarrow \mc$ be a meromorphic function on $T$.  We say that $f$ is polynomially bounded in vertical strips in $T$ if for every compact subdomain $S\subset T$, there are finitely many linear functions $\ell_1(s,w),\dots,\ell_n(s,w)$ such that the function $f(s,w)\ell_1(s,w)\dots \ell_n(s,w)$ is holomorphic and polynomially bounded in vertical strips in $S$.
\end{defin}

Using \cite[Lemmas 2.12--2.13]{Cech3}, we have that if a meromorphic function of two variables is polynomially bounded in vertical strips on a tube domain $T$, then so is its continuation to $\ch(T)$. 

\begin{proposition}\label{le:bounded in vertical strips}
	Let $R_1,R_2\subset \mc^2$ be open tube domains such that $R_1\cap R_2\neq \varnothing$ and $f(s,w)$ be a function that is holomorphic or meromorphic on $R_1,R_2$ and polynomially bounded in vertical strips there. Then it is also polynomially bounded in vertical strips in $\ch(R_1 \cup R_2)$.
\end{proposition}

\section{Proof of Theorem \ref{thm:MDS}}
\label{sec:proofofmainthm}

\subsection{First region of convergence}
\label{sec:basic}
We establish some the analytic properties of $A(s,w)$, defined in \eqref{Asw}. 
\begin{proposition}\label{prop:basic properties of A}
	With the notation as above, the function $A(s,w)$ has the following properties.
	\begin{enumerate}
		\item We have the functional equation 
		\begin{equation*}
			A(s,w)=X_K(s)A(1-s,s+w-1/2),
		\end{equation*}
where $X_K(s)$ is defined in \eqref{XK}. 
	\item The function $A(s,w)$ has meromorphic continuation to the region
\begin{align*}
  R_1=\{ (s,w) \in \comc^2: \Re(w)>0,\ \Re(s+w)>3/2\}.
\end{align*}
	The only pole of $\zeta_K(2w)A(s,w)$ in $R_1$ is a simple pole at $w=1$. 
	\item The function $(w-1)\zeta_K(2w) A(s,w)$ is polynomially bounded in vertical strips in $R_1$.
	\end{enumerate} 
\end{proposition}
\begin{proof} The assertion in (1) follows by applying the functional equation \eqref{fneqnL} of $L(s,\chi_{(1+i)^5d})$ in \eqref{Asw}, and the fact that all the characters $\chi_{(1+i)^5d}$ are primitive modulo $(1+i)^5d$ for odd, square-free $d$. \newline

To prove (2), we first note that by Lemma \ref{le:lindelof on average} and partial summation that $A(s,w)$ is holomorphic in the region 
	$$
		R_{1,1}=\{ (s,w) \in \comc^2: \Re(w)>1,\ \Re(s+w)>3/2\}.
	$$
Now \eqref{quadrec} leads to 
\begin{align*}
		A(s,w)= \sumstar_{\substack{ d\odd}}\sum_{n\odd}\frac{\chi_{(1+i)^5d}(n)}{N(d)^wN(n)^s}=\sum_{n\odd }\frac{\leg{(1+i)^5}{n}}{N(n)^s}\sum_{d\odd }\frac{\leg{d}{n}\mu^2_{[i]}(d)}{N(d)^w}
	\end{align*}
  
Re-writing the $d$-sum in the last expression above into its Euler product, the above can be recast as
\begin{equation}
\label{eq:second expr for A}
	\begin{aligned}
		A(s,w)=\sum_{n\odd}\frac{\leg{1+i}{n}L^{(2)}\lz w,\chi_{n}\pz}{N(n)^sL^{(2)}\lz 2w,\chi_{n}^2\pz}=\frac{1}{\zeta_K(2w)}\sum_{n\odd }\frac{\leg{1+i}{n}L^{(2)}\lz w,\chi_{n}\pz a_{2w}(2n)}{N(n)^s},
	\end{aligned}
\end{equation}
   where the function $a_w(n)$ is defined by 
\begin{equation}\label{eq:def of a_t(n)}
	a_w(n)=\prod\limits_{\varpi|n}\lz 1-\frac{1}{N(\varpi)^w}\pz^{-1}.
\end{equation} 
  Here we have
\begin{equation}
\label{eq:bound for a_t(n)}
	|a_{w}(n)|\ll_w
 N(n)^\varepsilon, \quad \mbox{for} \quad \Re(w)>0.
\end{equation}

  By Lemma \ref{le:lindelof on average} and \eqref{eq:bound for a_t(n)}, the last series in \eqref{eq:second expr for A} converges absolutely in the region
$$
	R_{1,2}=\{\Re(w)>0,\ \Re(s)>1,\ \Re(s+w)>3/2\},
$$
save for a pole at $w=1$ from $L^{(2)}\lz w,\chi_{n}\pz$ when $n=\square$ and primary.  Here and after, $\square$ denotes a perfect square in $\mathcal O_K$.  
By Theorem \ref{thm:bochner}, we see that $A(s,w)$ has a meromorphic continuation to the region
$$
	R_1=\mathrm{conv}(R_{1,1} \cup R_{1,2})=\{\Re(w)>0,\ \Re(s+w)>3/2\}.
$$
 The function $A(s,w)$ has no pole in $R_{1,1}$ and from \eqref{eq:second expr for A}, the only pole of $\zeta_K(2w)A(s,w)$ in $R_{1,2}$, and hence also in $R_1$, is at $w=1$.  \newline

Finally, the assertion in (3) follows from Proposition \ref{le:bounded in vertical strips} and the fact that $A(s,w)$ is polynomially bounded in vertical strips in both $R_{1,1}$ and $R_{1,2}$.
\end{proof}

\subsection{Second region of convergence}
\label{sec:fe in w}

We continue our analysis of $A(s,w)$ using the functional equation in $w$.  From \eqref{eq:second expr for A}, we write $A(s,w)$ as
\begin{align}
\label{Adecomp}
\begin{split}
		A(s,w)=& A_1(s,w)+A_2(s,w),
\end{split}
\end{align}
  where
\begin{align}
\label{A1A2def}
\begin{split}
		A_1(s,w)=& \frac{1}{\zeta_K(2w)}\sum_{\substack{n\odd \\ n=\square} }\frac{\leg{1+i}{n}L^{(2)}\lz w,\chi_{n}\pz a_{2w}(2n)}{N(n)^s} \quad \mbox{and} \quad
  A_2(s,w)= \frac{1}{\zeta_K(2w)}\sum_{\substack{n\odd \\ n \neq \square} }\frac{\leg{1+i}{n}L^{(2)}\lz w,\chi_{n}\pz a_{2w}(2n)}{N(n)^s}.
\end{split}
\end{align}

   We simplify first the expressions for $A_1(s,w)$ and $A_2(s,w)$ by noticing that $L^{(2)}\lz w,\chi_{n}\pz=L(w,\widetilde{\chi_n})$, so that
\begin{align}
\label{eq:second expr for A1}
\begin{split}
		A_1(s,w)=&\frac{1}{\zeta_K(2w)}\sum_{\substack{n\odd \\ n=\square} }\frac{\zeta_K(w)\prod_{\varpi|2n}\lz1-\frac{1}{N(\varpi)^w}\pz a_{2w}(2n)}{N(n)^s} \\ 
		 = & \frac{\zeta_K(w)}{\zeta_K(2w)}\sum_{\substack{n\odd } }\frac{\prod_{\varpi|2n}\lz1+\frac{1}{N(\varpi)^w}\pz^{-1} }{N(n)^{2s}} = \frac{\zeta_K(w)}{\zeta_K(2w)}\zeta_K^{(2)}(2s)B(s,w), \quad \mbox{and} \\
  A_2(s,w)=& \frac{1}{\zeta_K(2w)}\sum_{\substack{n\odd \\ n \neq \square} }\frac{\leg{1+i}{n}L(w,\widetilde{\chi_n}) a_{2w}(2n)}{N(n)^s},
\end{split}
\end{align}
  where
\begin{align*}
\begin{split}
		B(s,w)= \prod_{\varpi \odd }\lz 1-\frac{1}{N(\varpi)^{2s}(N(\varpi)^w+1) }\pz.
\end{split}
\end{align*}
Note that we easily recover the penultimate expression for $A_1(s,w)$ in \eqref{eq:second expr for A1} by multiplying out the Euler product of $\zeta_K^{(2)}(2s)$ with $B(s,w)$ defined above. \newline

 We readily see from \eqref{eq:second expr for A1} that $A_1(s,w)$ is meromorphic in the region 
\begin{align}
\label{R2}
  R_2=\{(s,w) \in \comc^2: \Re(s)>1,\ \Re(w)>0\}.
\end{align}
For $A_2(s,w)$, we recall the definition of $a_t(n)$ from \eqref{eq:def of a_t(n)} and by applying the functional equation given in Proposition \ref{Functional equation with Gauss sums} to see that 
\begin{align}
\label{A2B}
			A_2(s,w)=& \frac {Y(w)}{\zeta_K(2w)} \Big(1-\frac 1{2^{2w}} \Big)^{-1}B(s+w-\tfrac12, 1-w; 2w).  
\end{align}
  where
\begin{align}
\label{B12}
\begin{split}
			B(s, w; t)=& \sum_{\substack{n\odd \\ n \neq \square} }\sum_{\substack{ k \neq 0 \\ k \in
   \mathcal{O}_K}}\frac { \leg{1+i}{n} g(k,\widetilde{\chi_n})N(n)^{-1/2}a_{t}(n)}{N(n)^{s}N(k)^{w}} \\
=& \sum_{\substack{n\odd } }\sum_{\substack{ k \neq 0 \\ k \in
   \mathcal{O}_K}}\frac { \leg{1+i}{n} g(k,\widetilde{\chi_n})N(n)^{-1/2}a_{t}(n)}{N(n)^{s}N(k)^{w}}-\sum_{\substack{n\odd \\ n =\square} }\sum_{\substack{ k \neq 0 \\ k \in
   \mathcal{O}_K}}\frac {g(k,\widetilde{\chi_n})N(n)^{-1/2}a_{t}(n)}{N(n)^{s}N(k)^{w}}  \\
:= & B_1(s,w;t)-B_2(s,w;t), \quad \mbox{say}.  
\end{split}
\end{align}

  It remains to analyze $B(s, w; t)$. To that end, we first apply Lemma \ref{Gausssum} and argue as in the proof of \cite[Lemma 4.4]{Cech3}.  Thus we arrive at the following result on the analytical property of a Dirichlet series involving with Gauss sums. 
\begin{lemma}
\label{le:continuation}
  With the notation as above, for any $N\in \mathbb{N}$, $\Re(s),\Re(t)>0$, and $\alpha \in \text{CG}$, we have
	\begin{equation*}
\begin{aligned}
			\sum_{\substack{n\odd } }\frac {g(k,n)a_{t}(n)\chi_{\alpha}(n)}{N(n)^{s+1/2}} =& \prod_{j=0}^N\frac{L\lz s+jt,\chi_{(1+i)^2i\alpha k} \pz}{L\lz 2s+2jt, \chi_{(1+i)^2i\alpha k}^2\pz}\cdot E_N\lz s,t; \chi_{(1+i)^2i\alpha k} \pz
		P(s,t;k, \chi_{\alpha}),
\end{aligned}
	\end{equation*}
where 
$$
	E_N\lz s,t; \chi_{(1+i)^2i\alpha k} \pz=\prod_{\varpi} \lz1+O_N\lz\frac{1}{N(\varpi)^{s+(N+1)t}}+\frac{1}{N(\varpi)^{2s+t}}\pz\pz,
$$
and 
\begin{equation}\label{eq:def of P}
	P(s,t;k,\chi_{\alpha})=\prod_{\substack{\varpi|k \\ \varpi \odd}}\lz 1+a_t(\varpi)\sum_{j\geq 1}\frac{g\lz k,\varpi^j\pz \chi_{\alpha}(\varpi^j)}{N(\varpi)^{js+j/2}}\pz
\end{equation}
 is a finite Euler product, which for any $\delta,\varepsilon>0$, $\Re(t)>\delta$, and $\Re(s)>1/2+\delta$ satisfies $|P(s,t;k, \chi_{\alpha})|\ll_{\varepsilon,\delta} N(k)^\varepsilon$.
\end{lemma}

   With Lemma \ref{le:continuation} at our disposal, we proceed to establish some analytic properties of $B(s, w; t)$. 
\begin{proposition}
\label{prop:continuation of B}
	With the notation as above and the truth of the generalized Lindel\"of Hypothesis, the function $B(s,w;t)$ satisfies the following.
		\begin{enumerate}
			\item It has a meromorphic continuation to the region
$$
				S=\{ (s,w, t)\in \comc^3 : \Re(s)>1/2,\ \Re(s+w)>3/2,\ \Re(t)>0\}.
$$
			\item Its only potential poles in $S$ are at $w=1$ and $s+tj=1$ for $j\in \mz_{\geq 0}$.
			
			\item  It is polynomially bounded in vertical strips in $S$.
		\end{enumerate} 
\end{proposition}
\begin{proof}
  By the first equality in \eqref{B12}, we see via \eqref{eq:bound for a_t(n)} together with Lemma \ref{le:lindelof for K} and partial summation that $B(s, w; t)$ is holomorphic in the region:
\begin{align*}
			S_1:=\{ (s,w, t)\in \comc^3 : \Re(s)>1, \Re(s+w)>3/2, \ \Re(t)>0 \}.
\end{align*}

   Further, we may recast $B_2(s,w;t)$ as
\begin{align}
\label{B2}
			B_2(s, w; t)=& \sum_{\substack{n\odd } }\sum_{\substack{ k \neq 0 \\ k \in
   \mathcal{O}_K}}\frac {g(k,\widetilde{\chi_{n^2}})N(n^2)^{-1/2}a_{t}(n)}{N(n)^{2s}N(k)^{w}}.  
\end{align}
  Similar to \eqref{gsumbound}, we see that for $\Re(w)>1$,
\begin{align*}
			\sum_{\substack{ k \neq 0 \\ k \in
   \mathcal{O}_K}}\frac {g(k,\widetilde{\chi_{n^2}})N(n^2)^{-1/2}}{N(k)^{w}} \ll N(n)^{\varepsilon}.  
\end{align*}

   We then deduce from \eqref{eq:bound for a_t(n)} and the above that $B_2(s,w;t)$ is holomorphic in the region
\begin{equation}
\label{eq:region N}
	S_2 :=\{ (s,w, t)\in \comc^3 : \Re(w)>1,\ \Re(s)>1/2,\  \Re(t)>0\}.
\end{equation} 
Turning our attention to $B_1(s,w;t)$, we first apply \eqref{g2exp} to evaluate $g\lz k, \widetilde\chi_{n}\pz$ and apply \eqref{supprule} to detect whether $n$ is of type $1$ or not using the character sums $\frac 12 (\chi_1(n) \pm \chi_{i}(n))$. This reveals that
\begin{align}
\label{B1exp}
\begin{split}
			B_1(s,w;t)=&  \sum_{\substack{n\odd } }\sum_{\substack{ k \neq 0 \\ k \in
   \mathcal{O}_K}}\frac {(-1)^{\Im (k)}g(k,n)a_{t}(n)\chi_{i(1+i)}(n)}{N(n)^{s+1/2}N(k)^{w}}+\sum_{\substack{n\odd } }\sum_{\substack{ k \neq 0 \\ k \in
   \mathcal{O}_K}}\frac {(-1)^{\Re (k)}g(k,n)a_{t}(n)\chi_{1+i}(n)}{N(n)^{s+1/2}N(k)^{w}} \\
=&  \sum_{\substack{ k \neq 0 \\ k \in
   \mathcal{O}_K}}\frac {(-1)^{\Im (k)}}{N(k)^{w}}\sum_{\substack{n\odd } }\frac {g(k,n)a_{t}(n)\chi_{i(1+i)}(n)}{N(n)^{s+1/2}}+\sum_{\substack{ k \neq 0 \\ k \in
   \mathcal{O}_K}}\frac {(-1)^{\Re (k)}}{N(k)^{w}}\sum_{\substack{n\odd }}\frac {g(k,n)a_{t}(n)\chi_{1+i}(n)}{N(n)^{s+1/2}}.
\end{split}
\end{align}

 For any $\delta>0$ and $N\in\mathbb{N}$, it now follows from \eqref{B1exp}, Lemma \ref{le:continuation} and the generalized Lindel\"of Hypothesis that $B_{1}(s,w;t)$ has meromorphic continuation to the region  
$$
	S_2(N,\delta)=\{(s,w, t)\in \comc^3 : \Re(w)>1,\ \Re(s)>1/2+\delta,\ \Re(t)>\delta,\ \Re(s+(N+1)t)>1\}.
$$
Since for any $\delta>0$, the intersection $\bigcap_{N=1}^\infty S_2(N,\delta)$ is non-empty, we can obtain a meromorphic continuation to the union of these regions, which equals
\begin{equation*}
	\bigcup\limits_{\delta>0,N\in \mathbb{N}}S_2(N,\delta)=S_2.
\end{equation*} 

 Recall that $B_{2}(s,w;t)$ is holomorphic in the region $S_2$.  It follows from this and \eqref{B12} that $B(s,w;t)$ has meromorphic continuation to the region $S_2$. Note that it is also holomorphic in the region $S_1$.  Then Theorem \ref{thm:bochner} implies that $B(s,w;t)$ has meromorphic continuation to the region 
$$
	S=\mathrm{conv}(S_1\cup S_2)=\{ (s,w, t)\in \comc^3 : \Re(s)>1/2,\ \Re(s+w)>3/2,\ \Re(t)>0\}.
$$
 The above implies (1). Now using Lemma \ref{le:continuation} and arguing as in the proof of \cite[Proposition 4.6]{Cech3} leads to (2) and (3). This completes the proof.
\end{proof}

\subsection{Completion of the proof}
\label{sec:residues}

We first prove the assertion in  (1).  We infer from \eqref{eq:second expr for A1} and Proposition \ref{prop:continuation of B} that $A_2(s,w)$ has a meromorphic continuation to the region $R_2$, with $R_2$ is defined in \eqref{R2}. Recall that $A_1(s,w)$ also has a meromorphic continuation to the region $R_2$. Then \eqref{Adecomp} implies that $A(s,w)$ also has meromorphic continuation to the region $R_2$.  From this, Proposition \ref{prop:basic properties of A}, Proposition \ref{prop:continuation of B} and arguing as in the proof of \cite[Proposition 4.6]{Cech3} using Theorem \ref{thm:bochner}, we arrive at the meromorphic continuation, location of poles of $A(s,w)$ and the polynomial boundedness in (1). \newline

Moving to (2), we note by \eqref{Adecomp} and \eqref{A1A2def} that the residue comes at $w=1$ comes from $A_1(s,w)$ only. 
 Recall that the residue of $\zeta_K(s)$ at $s=1$ equals $\pi/4$, so that we deduce readily from \eqref{eq:second expr for A1} the corresponding residue. \newline

For (3), the residue at $w=3/2-s$ comes from the functional equation in Proposition \ref{prop:basic properties of A} (1). \newline

Now for (4), the functional equation $A(s,w)=X_K(s)A(1-s,s+w-1/2)$ transforms the pole at $w=2/3-2s/3$ into the pole at $s+w-1/2=2/3-2(1-s)/3$, or equivalently, at $w=1/2-s/3$, and a computation similar to that of (2) shows that
$$
	\res_{w=2/3-2s/3}A(s,w)=X_K(s)\res_{w=1/2-s/3} A(1-s,w).
$$ 

Finally, it still remains to prove (5).  As a preparation, we first evaluate $\res_{t=1-s}B(s,w;t)$ in the following lemma.
\begin{lemma}\label{le:R(s,w,psi,psi')}
	For $0<\Re(s)<1$, $\Re(w)>1/2$ and $\Re(s+w)>1$, we have
\begin{align*}
\begin{split}
			\res_{t=1-s}B(s,w;t) =& \frac{\pi \zeta_K(s)E_K(s)}{\zeta_K(2s)\zeta_K(2)}\lz 1+\frac{2}{2^s(2-2^s)}\pz^{-1} \frac 1{2^w} \lz \lz1-\frac{1}{2^{2w}}\pz^{-1} -2\pz P_K(s,w). 
\end{split}
\end{align*}
  where $E_K(s)$ is defined in \eqref{EKdef} and 
\begin{align}
\label{EP}
\begin{split}
		P_K(s,w) =&\zeta^{(2)}_K(2s+2w-1)\prod_{\substack{\varpi \\ \varpi \odd}}\lz 1-\frac{N(\varpi)^{1+2s}+N(\varpi)^{1+s}-N(\varpi)-N(\varpi)^{3s}}{N(\varpi)^s(-N(\varpi)^{1+s}-N(\varpi)+N(\varpi)^{2s})(N(\varpi)^{2w}-1)}\pz.
\end{split}
\end{align}
\end{lemma}
\begin{proof}
	We apply Lemma \ref{Gausssum} and argue as in the proof of \cite[Lemma 4.4]{Cech3} to evaluate $B_2(s,w;t)$ given in \eqref{B2} to see that it has no pole at $t=1-s$. It follows from \eqref{B12} that $\res_{t=1-s}B(s,w;t)=\res_{t=1-s}B_1(s,w;t)$. We then deduce from \eqref{B1exp} and Lemma \ref{le:continuation} with $N=1$ that
	\begin{equation}\label{eq:expanded B}
\begin{aligned}
		B_1(s,w;t)
=&  \sum_{\substack{ k \neq 0 \\ k \in
   \mathcal{O}_K}}\frac {(-1)^{\Im (k)}}{N(k)^{w}}\sum_{\substack{n\odd } }\frac{L\lz s,\chi_{(1+i)k} \pz  L\lz s+t,\chi_{(1+i)k} \pz}{L\lz 2s, \chi_{(1+i)k}^2\pz  L\lz 2s+2t, \chi_{(1+i)^2k}^2\pz } E_1\lz s,t; \chi_{(1+i)k} \pz
		P(s,t;k, \chi_{i(1+i)}) \\
& \hspace*{1cm} +\sum_{\substack{ k \neq 0 \\ k \in
   \mathcal{O}_K}}\frac {(-1)^{\Re (k)}}{N(k)^{w}}\sum_{\substack{n\odd } }\frac{L\lz s,\chi_{i(1+i)k} \pz L\lz s+t,\chi_{i(1+i)k} \pz}{L\lz 2s, \chi_{i(1+i)k}^2\pz  L\lz 2s+2t, \chi_{i(1+i)k}^2\pz } E_1\lz s,t; \chi_{i(1+i)k} \pz P(s,t;k, \chi_{1+i}) \\
:= & B_{1,1}(s,w;t)+B_{1,2}(s,w;t). 
\end{aligned}
\end{equation}
Hence $B_{1,1}(s,w;t)$ (resp. $B_{1,2}(s,w;t)$) has a simple pole at $t=1-s$ when $k=(1+i)\square$ (resp. $k=i(1+i)\square$). Recall that the residue of $\zeta_K(s)$ at $s=1$ is $\pi/4$.  We apply this and argue as in the proof of \cite[Lemma 5.2]{Cech3}.  Thus, if $k=(1+i)\square$, then
\begin{align}
\label{B1res}
\begin{split}
	& \res_{t=1-s}B_{1,1}(s,w;t) \\
= & \frac 1{2^w} \sum_{\substack{ k \neq 0 \\ k^2 \in
   \mathcal{O}_K}}\frac {(-1)^{\Im ((1+i)k^2)}}{N(k)^{2w}} \frac{\zeta^{(2k)}_K(s)\res_{t=1-s}\zeta^{(2k)}_K(s+t)}{\zeta^{(2k)}_K(2s)\zeta^{(4k)}_K(2)}E_1\lz s,1-s;\chi_{(1+i)^2k^2}\pz P(s,1-s; (1+i)k^2, \chi_{i(1+i)})\\
			=& \frac {\pi}{4} \frac{\zeta_K(s)E_K(s)}{\zeta_K(2s)\zeta_K(2)}\frac 1{2^w} \sum_{\substack{ k \neq 0 \\ k^2 \in
   \mathcal{O}_K}}\frac {(-1)^{\Im ((1+i)k^2)}P(s,1-s; (1+i)k^2, \chi_{i(1+i)})}{N(k)^{2w}}\prod_{\varpi |2k}\lz1+\frac{N(\varpi)}{N(\varpi)^s(N(\varpi)-N(\varpi)^s)}\pz^{-1}. 
\end{split}
\end{align}
	
  We further note that by \eqref{eq:def of P} and Lemma \ref{Gausssum}, 
\begin{align}
\label{Pexp}
\begin{split}
		P(s,1-s; (1+i)k^2, \chi_{i(1+i)}) =&  \prod_{\substack{ \varpi |k \\ \varpi \odd}}\lz 1+\lz1-\frac{1}{N(\varpi)^{1-s}}\pz^{-1}\sum_{j\geq 1}\frac{g\lz (1+i)k^2, \varpi^j \pz \chi_{i(1+i)}(\varpi^j)}{N(\varpi)^{js+j/2}}\pz\\
		=&  \prod_{\substack{ \varpi^{\alpha} \| k \\ \varpi \odd}} \Bigg( 1+\frac{N(\varpi)-1}{(N(\varpi)^s-N(\varpi))(1-N(\varpi)^{2s-1})} \\
& \hspace*{1.5cm} +\frac{1}{N(\varpi)^{\alpha s-\alpha/2}}\Big( \frac{1}{N(\varpi)^s(1-N(\varpi)^{s-1})}+\frac{N(\varpi)-1}{(N(\varpi)-N(\varpi)^s)(1-N(\varpi)^{2s-1})}\Big) \Bigg).
\end{split}
\end{align}

  Similarly, when $k=i(1+i)\square$,
\begin{align*}
\begin{split}
	& \res_{t=1-s}B_{1,2}(s,w;t) \\
			=& \frac {\pi}{4} \frac{\zeta_K(s)E_K(s)}{\zeta_K(2s)\zeta_K(2)} \frac 1{2^w} \sum_{\substack{ k \neq 0 \\ k^2 \in
   \mathcal{O}_K}}\frac {(-1)^{\Re (i(1+i)k^2)}P(s,1-s; i(1+i)k^2, \chi_{1+i})}{N(k)^{2w}}\prod_{\varpi |2k}\lz1+\frac{N(\varpi)}{N(\varpi)^s(N(\varpi)-N(\varpi)^s)}\pz^{-1}. 
\end{split}
\end{align*}

  A direct computation shows that $P(s,1-s; i(1+i)k^2, \chi_{1+i})=P(s,1-s; (1+i)k^2, \chi_{i(1+i)})$. Note $(-1)^{\Re (i(1+i)k^2)}=(-1)^{\Im ((1+i)k^2)}$.  Hence
\begin{align*}
\begin{split}
	\res_{t=1-s}B_{1,2}(s,w;t)
			=& \res_{t=1-s}B_{1,1}(s,w;t). 
\end{split}
\end{align*} 

  We then deduce from \eqref{eq:expanded B}, \eqref{B1res} and the above that
\begin{align*}
\begin{split}
\res_{t=1-s} & B_{1}(s,w;t) =  2\res_{t=1-s}B_{1,1}(s,w;t) \\
=&  \frac {\pi}{2} \frac{\zeta_K(s)E_K(s)}{\zeta_K(2s)\zeta_K(2)} \frac 1{2^w} \sum_{\substack{ k \neq 0 \\ k^2 \in
   \mathcal{O}_K}}\frac {(-1)^{\Im ((1+i)k^2)}P(s,1-s; (1+i)k^2, \chi_{i(1+i)})}{N(k)^{2w}}\prod_{\varpi |2k}\lz1+\frac{N(\varpi)}{N(\varpi)^s(N(\varpi)-N(\varpi)^s)}\pz^{-1}.		 
\end{split}
\end{align*} 

   We further note that any $k^2 \in \mathcal O_K, k \neq 0$ can be written as $\pm (1+i)^{2j}k^2_1$ with $j \geq 0$ and $k_1$ being primary. Recall that $k_1=a+bi \in \mathcal{O}_K$ with $a, b \in \mz$ is primary if and only if $a \equiv 1 \pmod{4}, b \equiv
0 \pmod{4}$ or $a \equiv 3 \pmod{4}, b \equiv 2 \pmod{4}$. The above observations imply that $(-1)^{\Im ((1+i)k^2)}=1$ when $j \geq 1$ while  $(-1)^{\Im ((1+i)k^2)}=-1$ when $j = 0$. Moreover, we have $P(s,1-s; (1+i)^jk^2_1, \chi_{i(1+i)})=P(s,1-s; -(1+i)^jk^2_1, \chi_{i(1+i)})$. Thus
\begin{align*}
\begin{split}
	&\res_{t=1-s}B_{1}(s,w;t) \\
=&  \frac{\pi\zeta_K(s)E_K(s)}{\zeta_K(2s)\zeta_K(2)}  \frac 1{2^w}\sum_{\substack{ j \geq 1 \\ k_1 \odd }}\frac {P(s,1-s; (1+i)^{2j+1}k^2_1, \chi_{i(1+i)})}{2^{2jw}N(k_1)^{2w}}\prod_{\varpi |2k}\lz1+\frac{N(\varpi)}{N(\varpi)^s(N(\varpi)-N(\varpi)^s)}\pz^{-1} \\
& -\frac{\pi\zeta_K(s)E_K(s)}{\zeta_K(2s)\zeta_K(2)}  \frac 1{2^w} \sum_{\substack{  k_1 \odd }}\frac {P(s,1-s; (1+i)k^2_1, \chi_{i(1+i)})}{N(k_1)^{2w}}\prod_{\varpi |2k}\lz1+\frac{N(\varpi)}{N(\varpi)^s(N(\varpi)-N(\varpi)^s)}\pz^{-1}		 
\end{split}
\end{align*} 
  
  Note by \eqref{Pexp} that we have $P(s,1-s; (1+i)^{2j+1}k^2_1, \chi_{i(1+i)})=P(s,1-s; (1+i)k^2_1, \chi_{i(1+i)})$ for all $j \geq 1$. This allows us to recast the expression above into an Euler product. By a direct computation, we obtain an expression of $\res_{t=1-s}B_{1}(s,w;t)$ that is similar to the expression $R(s,w;\psi_1,\psi_1)$ given in the proof of \cite[Lemma 5.2]{Cech3}. Therefore,
$$
	\res_{t=1-s}B_{1}(s,w;t)=\frac{\pi \zeta_K(s)E_K(s)}{\zeta_K(2s)\zeta_K(2)}\lz1+\frac{2}{2^s(2-2^s)}\pz^{-1} \frac 1{2^w} \lz \lz1-\frac{1}{2^{2w}}\pz^{-1} -2\pz P_K(s,w),
$$ 
where $P_K(s,w) =\prod_{\substack{\varpi \\ \varpi \odd}}P_{\varpi, K}(s,w)$ with
\begin{align*}
\begin{split}
	P_{\varpi, K} & (s,w) \\
=& 1+ \frac{N(\varpi)(N(\varpi)^{1+s}-N(\varpi)^{2s}+N(\varpi))+N(\varpi)^{2w}(-N(\varpi)^{2+s}-N(\varpi)^2-N(\varpi)^{1+3s}+N(\varpi)^{1+s}
+N(\varpi)^{4s})}{(N(\varpi)^{2s}-N(\varpi)^{1+s}-N(\varpi))(N(\varpi)^{2w}-1)(N(\varpi)^{2s+2w}-N(\varpi))}.
\end{split}
\end{align*} 

   We further simplify the above expression for $P_K(s,w)$ similar to what is done on \cite[p.33]{Cech3} for the function $P(s,w)$ there.  The upshot is that
our $P_K(s,w)$ equals the expression given in \eqref{EP}. Note that both $E_K(s)$ and $P_K(s,w)$ are absolutely convergent in the region under consideration.
\end{proof}

  Now, to compute the residue at $w=1/2-s/3$ of $A(s,w)$, we apply the functional equation \eqref{A2B}.  We deduce that this residue comes from the poles of $B(s,w;t)$ at $s+t=1$, so that
\begin{align}
\label{Rdef}
\begin{split}
			R(s):=& \res_{w=1/2-s/3}A(s,w)=\res_{w=1/2-s/3}A_2(s,w) \\
=&\res_{w=1/2-s/3}\frac {Y(\frac 12-\frac s3)}{\zeta_K(1-\tfrac{2s}{3})}\Big(1-\frac 1{2^{2(1/2-s/3)}}\Big)^{-1}B(s+w-\tfrac12, 1-w; 2w) \\
=& \frac {Y(\frac 12-\frac s3)}{\zeta_K(1-\tfrac{2s}{3})} \Big(1-\frac 1{2^{2(1/2-s/3)}}\Big)^{-1}
\res_{t=1-2s/3}B(\tfrac{2s}{3},\tfrac12+\tfrac s3; t).
\end{split}
\end{align}

   We now deduce from \eqref{Rdef} and Lemma \ref{le:R(s,w,psi,psi')} that
\begin{align*}
\begin{split}
			R(s) 
=& \frac {\pi Y(\frac 12-\frac s3)}{\zeta_K(1-\tfrac{2s}{3})}\Big(1-\frac 1{2^{2(1/2-s/3)}}\Big)^{-1}
\res_{t=1-2s/3}B(\tfrac{2s}{3},\tfrac12+\tfrac s3; t)\\
=& \frac {\pi Y(\frac 12-\frac s3)}{\zeta_K(1-\tfrac{2s}{3})}\frac{\zeta_K\bfrac{2s}3E_K\bfrac{2s}3}{\zeta_K\bfrac{4s}3\zeta_K(2)}\lz 1-\frac 1{2^{2(1/2-s/3)}} \pz^{-1}\lz1+\frac{2}{2^{2s/3}(2-2^{2s/3})}\pz^{-1}\frac{1}{2^{1/2+s/3}} \lz \lz 1-\frac{1}{2^{1+2s/3}}\pz^{-1}-2 \pz \\
& \times  P_K\lz\tfrac{2s}{3},\tfrac{3+2s}{6}\pz \\
=& \frac {\pi Y(\frac 12-\frac s3)}{\zeta_K(1-\tfrac{2s}{3})}\frac{\zeta_K\bfrac{2s}3E_K\bfrac{2s}3}{\zeta_K\bfrac{4s}3\zeta_K(2)} \lz 1-\frac 1{2^{2(1/2-s/3)}} \pz^{-1}\lz1+\frac{2}{2^{2s/3}(2-2^{2s/3})}\pz^{-1}\frac{1}{2^{1/2+s/3}} \lz \lz 1-\frac{1}{2^{1+2s/3}}\pz^{-1}-2 \pz \\
& \hspace*{1cm} \times \zeta^{(2)}_K(2s)
\prod_{\substack{\varpi \\ \varpi \odd}}
\lz1+\frac{N(\varpi)^{1+4s/3}+N(\varpi)^{1+2s/3}-N(\varpi)-N(\varpi)^{2s}}{N(\varpi)^{2s/3}(N(\varpi)^{1+2s/3}+N(\varpi)-N(\varpi)^{4s/3})(N(\varpi)^{1+4s/3}-1)}\pz \\
=& \frac{\pi Y(\frac {2s}3)Y(\frac 12-\frac s3)\zeta_K(2s)E_K\bfrac{2s}3}{\zeta_K\bfrac{4s}3\zeta_K(2)}2^{2+2s}\lz1-\frac{1}{2^{2s}}\pz \lz 1-\frac 1{2^{2(1/2-s/3)}} \pz^{-1} \lz1+\frac{2}{2^{2s/3}(2-2^{2s/3})}\pz^{-1} \\
& \times \frac{1}{2^{1/2+s/3}} \lz \lz 1-\frac{1}{2^{1+2s/3}}\pz^{-1}-2 \pz \\
& \times \prod_{\substack{\varpi \\ \varpi \odd}}
\lz1+\frac{N(\varpi)^{1+4s/3}+N(\varpi)^{1+2s/3}-N(\varpi)-N(\varpi)^{2s}}{N(\varpi)^{2s/3}(N(\varpi)^{1+2s/3}+N(\varpi)-N(\varpi)^{4s/3})(N(\varpi)^{1+4s/3}-1)}\pz, 
\end{split}
\end{align*}
    where the last equality above follows by noting from \eqref{fneqnL} with $q=1$ there and \eqref{Ydef} that $\frac{\zeta_K(2s/3)}{\zeta_K(1-2s/3)}=2^{2+2s}Y(\frac {2s}3)$.  The product above is absolutely convergent for $0<\Re(s)<3/2$. The above readily leads to the expression for $\res_{w=1/2-s/3} A(s,w)$ given in (5). This completes the proof of Theorem \ref{thm:MDS}.

\section{Proof of Theorem \ref{Theorem for all characters}}
\label{sec:proof}

    We proceed from \eqref{Sumintegral} by shifting the integral therein to the line $c=\max\{\beta/2, (\beta+1)/2-\Re(s),1/2-\Re(s)/2\}+\varepsilon$, where $\beta$ is defined in \eqref{betadef}. By Theorem \ref{thm:MDS}, the possible poles of $A(s,w)$ occur at 
the zeros of $\zeta_K(2w), \zeta_K(2s+2w-1)$ together with those at $w=1$, $s+w=3/2$, $s+(2j+1)w=3/2$ and $2js+(2j+1)w=j+1$ for $j\in\mathbb{N}$.
Since $c \geq \max\{\beta/2,(\beta+1)/2-\Re(s)\}+\varepsilon$, there is no contribution from the possible poles from the zeros of $\zeta_K(2w), \zeta_K(2s+2w-1)$. 
 Moreover, similar to \cite[Lemma 12.6]{MVa1}, if $0\leq \Re(s) \leq 2$, we have
\begin{align*}
 \frac {\zeta'_K(s)}{\zeta_K(s)}  \ll \sum_{\substack{\rho \\ |\Im(\rho)-\Im(s)| \leq 1}}\frac 1{s-\rho}+O(\log (4+|\Im(s)|)),
\end{align*}
 where $\rho$ denotes a zero of $\zeta_K(s)$.  Next, using \cite[(5.27)]{iwakow}, we deduce from the above that when $\Re(s) \geq \beta+\varepsilon$,
\begin{align*}
 \frac {\zeta'_K(s)}{\zeta_K(s)}  =O_{\varepsilon}(\log (4+|\Im(s)|)). 
\end{align*}
 We then argue similar to the proof of \cite[(5.27)]{iwakow} to see that for $\Re(s) \geq \beta+\varepsilon$, we have
\begin{align*}
 \frac 1{\zeta_K(s)}  \ll (1+|s|)^{\varepsilon}.
\end{align*}
 It follows from the above and $c \geq \max\{\beta/2,(\beta+1)/2-\Re(s)\}+\varepsilon$ that $|\zeta_K(2w)\zeta_K(2s+2w-1)^{-1} \ll (1+|s|)^{\varepsilon}$ on the line $\Re(w)=c$.  As Theorem \ref{thm:MDS} implies that $\zeta_K(2w)\zeta_K(2s+2w-1)A(s,w)$ is bounded on the line $\Re(w)=c$ by $(1+|w|)^{D}$ for some positive real $D$, we deduce that $A(s,w)$ is also bounded on the line $\Re(w)=c$ by $(1+|w|)^{D'}$ for some positive real $D'$.  Note also that repeated integration by parts implies that for any integer $E \geq 0$,
\begin{align*}
 \widehat \Phi(s)  \ll  \frac{1}{(1+|s|)^{E}}.
\end{align*}
  This implies that the integral on the new line $\Re(w)=c$ is absolutely convergent and bounded by $O(X^{c+\varepsilon})$, which yields the desired $O$-term in
\eqref{eq:general moment result}. Similar to the arguments given in the proof of \cite[Theorem 1.2]{Cech3}, one shows that one encounters only the poles at $w=1$, $w=3/2-s$, $w=1/2-s/3$ and $w=2/3-2s/3$ in the process and no other poles as $\Re(s)>1/3$. Computing the corresponding residues using Theorem \ref{thm:MDS} now leads to the main terms in \eqref{eq:general moment result}. This completes the proof of Theorem \ref{Theorem for all characters}.

\vspace*{.5cm}

\noindent{\bf Acknowledgments.}  P. G. is supported in part by NSFC grant 12471003 and L. Z. by the FRG Grant PS71536 at the University of New South Wales.  The authors wish to thank the anonymous referee for his/her very careful inspection of the manuscript and many helpful comments.

\bibliography{biblio}
\bibliographystyle{amsxport}

\end{document}